\documentclass[letterpaper, 10 pt, conference]{ieeeconf}

\IEEEoverridecommandlockouts                              
\newtheorem{theorem}{Theorem}
\newtheorem{corollary}{Corollary}
\newtheorem{lemma}{Lemma}

\newtheorem{definition}{Definition}
\newtheorem{remark}{Remark}
\newtheorem{example}{Example}
\usepackage{cite}

\usepackage{graphics} 
\usepackage{epsfig} 
\usepackage{mathptmx} 
\usepackage{times} 
\usepackage{amsmath} 
\usepackage{amssymb}  

\title{\LARGE \bf
Partial-fraction Expansion of Lossless Negative Imaginary Property and A Generalized Lossless Negative Imaginary Lemma}

\author{Mei Liu and Gang Chen
\thanks{Mei Liu is with the Department of Mechanical Engineering,
The Hong Kong Polytechnic University, Hong Kong, China. Gang Chen is with the Department of Mechanical and Aerospace Engineering, University of California, Davis, CA. Email: {\tt\small Lmaymay@mail.ustc.edu.cn} (Mei Liu),
{\tt\small ggchen@ucdavis.edu} (Gang Chen). 
}}

\begin{document}

\maketitle
\thispagestyle{empty}
\pagestyle{empty}

\begin{abstract}
This paper studies a partial-fraction expansion for lossless negative imaginary systems and presents a generalized lossless negative imaginary lemma by allowing poles at zero. First, a necessary and sufficient condition for a system to be non-proper lossless negative imaginary is developed, and a minor partial-fraction expansion of lossless negative imaginary property is studied. Second, according to the minor decomposition properties, two different and new relationships between lossless positive real and lossless negative imaginary systems are established.
 Third, according to one of the relationships, a generalized lossless negative imaginary lemma in terms of a minimal state-space realization is derived by allowing poles at zero. Some important properties of lossless negative imaginary systems are also studied in this paper, and three numerical examples are provided to illustrate the developed theory.

\end{abstract}

\section{Introduction}

The concept of lossless is related to that of passivity \cite{anderson1973network}. An m-port network, assumed to be storing no energy at time $t$, is called lossless if it is passive and if, when a finite amount of energy is put into the elements, all the energy can be extracted again \cite{anderson1973network}. It is well known that systems which dissipate energy often result in positive real properties \cite{anderson1973network,brogliato2006dissipative}. The so-called lossless positive real systems are systems whose positive real transfer matrix $F(s)$
satisfies the condition: $F(j\omega)+F^{*}(j\omega)=0$ for all real $\omega$. That is, the negative of a lossless positive real transfer matrix is its own complex conjugate transpose \cite{newcomb1966linear}. More physical descriptions about the lossless positive real systems can be found in \cite{bitmead1977matrix,anderson1973network,newcomb1966linear}.

Since the concept of lossless positive real systems first appeared in \cite{newcomb1966linear}, the research of lossless positive real systems have attracted more attention among control theorist \cite{vaidyanathan1985discrete,kraus1988robust,bitmead1977matrix,buscarino2016invariance}. For example, article \cite{bitmead1977matrix} gave a matrix fraction description of lossless positive real property in terms of a Hankel matrix.
The invariance of characteristic values and $L_{\infty}$ norm under lossless positive real transformations were studied in \cite{buscarino2016invariance}.
The continuous-time and discrete-time lossless positive real lemma in terms of minimal state-space realization were developed in \cite{anderson1973network} and \cite{xiao1999generalizations,byrnes1994losslessness}, respectively. Also, it is noteworthy that lossless positive real transfer functions form a convex set, which showed an important role in the proof of Kharitonov's theorem, see \cite{kraus1988robust}.

One major limitation of (lossless) positive real system is that their relative degree must be zero or one \cite{brogliato2006dissipative}. Negative imaginary systems theory, which allowed the relative degree to be two, has appeared as a useful complement to passivity theory and positive real theory \cite{petersen2015physical,lanzon2008stability}. More research and applications on negative imaginary systems can be found in \cite{bhikkaji2012negative,cai2010stability,liu2015alpha,
bhowmick2017lti,wang2015robust,liu2017properties,mabrok2015generalized}. 
In this paper, we are interested in studying a special and important class of negative imaginary systems, that is, lossless negative imaginary systems.
Dynamical systems with lossless negative imaginary properties have much more applications in the control of undamped flexible structures and lossless electrical circuits, see \cite{anderson1973network,xiong2012lossless,newcomb1966linear}.
The definition of proper lossless negative imaginary systems was firstly proposed in \cite{xiong2012lossless} by restricting no poles at zero and infinity, and a minimal state-space characterization of such systems was also developed in \cite{xiong2012lossless}. Then, an algebraic approach to the realization of lossless negative imaginary systems was studied in \cite{rao2012algebraic}. Subsequently, article \cite{liu2016non} extended the definition of lossless negative imaginary systems to non-proper case by allowing poles at zero and infinity.

In this paper, we will further present a partial-fraction expansion for lossless negative imaginary systems. That is, the lossless negative imaginary systems can be decomposed as a sum of several lossless negative imaginary systems. Based on this minor decomposition properties, two new relationships between lossless positive real and lossless negative imaginary transfer matrices are established from proper and non-proper case, and a generalized lossless negative imaginary lemma in terms of minimal state-space condition is derived by allowing poles at zero.
The extended results of the paper show a nice parrel to the better understood results on non-proper lossless positive real systems.

The rest of the paper is organized as follows: Section \ref{sec:pre} reviews the definition of lossless positive real and lossless negative imaginary systems. Some useful preliminary results and the minor decomposition theory are also introduced in this section. Section \ref{sec:relation} establishes two relationships between lossless negative imaginary and lossless positive real transfer matrices. Based on one of the relationships, a generalized lossless negative imaginary lemma is developed in Section \ref{sec:lemma}.
Section \ref{sec:conclusions} concludes the paper.

%
$\textit{Notation:}$ $\mathbb{R}^{m\times n}$ and $\mathcal{R}^{m\times n}$ denote the sets of $m\times n$ real matrices and real-rational proper transfer matrices, respectively.
$A^{T}$, $\bar{A}$ and $A^{*}$ denote the transpose, the complex conjugate and the complex conjugate transpose of a complex matrix $A$, respectively.
$A>(\geq)0$ denotes a symmetric positive (semi-) definite matrix.

%
%

\section{Preliminary results}
\label{sec:pre}
The concept of lossless positive real and lossless negative imaginary systems are introduced in this section. Also, some useful properties for lossless positive real and lossless negative imaginary systems, which will be used to develop the main results of this paper, are stated in this section.
\subsection{Lossless positive real transfer function matrices}
This subsection introduces the concepts of positive real and
lossless positive real transfer matrices, and derives a generalized lossless positive real lemma by removing the observability assumption.
\begin{definition}
\label{def:1}
\cite{anderson1973network} A square transfer function matrix $F(s)$ is said to be positive real if
\begin{enumerate}
\item All elements of $F(s)$ are analytic in $\mathrm{Re}[s]>0$; \label{def:1-1}
\item $F(s)$ is real for real positive $s$;\label{def:1-2}
\item $F^{*}(s)+F(s)\geq0$ for $\mathrm{Re}[s]>0$.\label{def:1-3}
\end{enumerate}
\end{definition}

\begin{definition}\cite{anderson1973network}
\label{def:2}
A square real-rational transfer function matrix $F(s)$ is said to be lossless positive real if
\begin{enumerate}
 \item $F(s)$ is positive real;\label{def:2-1}
 \item $F(j\omega)+F^{*}(j\omega)=0$ for all $\omega>0$ except values of $\omega$ where $j\omega$ is a pole of $F(s)$.\label{def:2-2}
\end{enumerate}
\end{definition}

The following lemma extends the positive real lemma in \cite{anderson1973network} by relaxing the observability requirement of $(A,C)$ and the non-singularity requirement of $P$.
\begin{lemma}\cite{anderson1968algebraic}
\label{lem:1}
Let $(A,B,C,D)$ be a state-space realization of a real-rational proper transfer function matrix $F(s)\in\mathcal{R}^{m\times m}$, where $(A,B)$ is completely controllable, $(A,C)$ is not necessarily completely observable, $A\in\mathbb{R}^{n\times n}$, $B\in\mathbb{R}^{n\times m}$, $C\in\mathbb{R}^{m\times n}$, $D\in\mathbb{R}^{m\times m}$, and $m\leq n$. Then, $F(s)$ is positive real if and only if there exist real matrices $P=P^{T}\geq0$, $P\in\mathbb{R}^{n\times n}$, $L$ and $W$ such that
\begin{align}
\label{eq:4}
\begin{aligned}
PA+A^TP&=-LL^{T}\\
PB&=C^T-L^{T}W\\
D+D^{T}&=W^TW.
\end{aligned}
\end{align}
\end{lemma}

Similarly, we can extend the lossless positive real lemma in \cite[page 229]{anderson1973network} by relaxing the observability requirement of $(A,C)$ and the non-singularity requirement of $P$. Lemma \ref{lem:2} is useful in the proof of lossless negative imaginary lemma in Section \ref{sec:lemma}.

\begin{lemma}
\label{lem:2}
Let $(A,B,C,D)$ be a state-space realization of a square real-rational proper transfer function matrix $F(s)\in\mathcal{R}^{m\times m}$, where $(A,B)$ is completely controllable, $(A,C)$ is not necessarily completely observable, $A\in\mathbb{R}^{n\times n}$, $B\in\mathbb{R}^{n\times m}$, $C\in\mathbb{R}^{m\times n}$, $D\in\mathbb{R}^{m\times m}$, and $m\leq n$. Then, $F(s)$ is lossless positive real if and only if there exist real matrices $P=P^{T}\geq0$, $P\in\mathbb{R}^{n\times n}$, $L$ and $W$ such that
\begin{align}
\label{eq:5}
\begin{aligned}
PA+A^TP&=0\\
PB-C^T&=0\\
D+D^{T}&=0.
\end{aligned}
\end{align}
\end{lemma}
\begin{proof}
(Necessity) Suppose $F(s)$ is lossless positive real. It is also positive real. According to Lemma \ref{lem:1}, there exist real matrices $P=P^{T}\geq0$, $L$ and $W$ satisfying \eqref{eq:4}. Then, from the $L$ and $W$, we can construct that
\begin{align*}
F(s)+F^{*}(s)=&[W^{T}+B^{T}(-sI-A^{T})^{-1}L][W+L^{T}(sI-A)^{-1}B]\\
&+B^{T}(s^{*}I-A^{T})^{-1}P(sI-A)^{-1}B(s+s^{*}),
\end{align*}
and hence,
\begin{align}
\label{eq:6}
\begin{aligned}
F(s)&+F^{T}(-s)\\
&=[W^{T}+B^{T}(-sI-A^{T})^{-1}L][W+L^{T}(sI-A)^{-1}B]\\
&\triangleq \Delta^{T}(-s)\Delta(s).
\end{aligned}
\end{align}
Let $s=j\omega$. Condition 2 of Definition \ref{def:2} implies that
\begin{align*}
F(j\omega)&+F^{T}(-j\omega)= \Delta^{*}(j\omega)\Delta(j\omega)=0.
\end{align*}
Hence, $\Delta(j\omega)=W+L^{T}(j\omega I-A)^{-1}B=0$ for all real $\omega$. The rest proof is the same as the proof of lossless positive real lemma in \cite[Page 222]{anderson1973network}. Note that only the controllability of $(A,B)$ is used to prove $L=0$.

(Sufficiency) Suppose \eqref{eq:5} holds. Because \eqref{eq:5} is the same as \eqref{eq:4} with $L=0$ and $W=0$, it follows from Lemma \ref{lem:1} that $F(s)$ is at least positive real. Applying the spectral factorization in \eqref{eq:6}, we have $F(j\omega)+F^{*}(j\omega)=0$ for all real $\omega$. According to the Definition \ref{def:2}, $F(s)$ is lossless positive real.
\end{proof}

\begin{remark}
The proof of Lemma \ref{lem:2} is mainly motivated by Lemma \ref{lem:1} and the proof of lossless positive real lemma in
\cite[page 222]{anderson1973network}.
The details on the spectral factorization for positive real transfer matrix can be found in \cite[page 219]{anderson1973network}.
The main difference between the proof of Lemma \ref{lem:2} and the proof of lossless positive real lemma in
\cite[page 222]{anderson1973network} is that the positive real lemma used in \cite[page 222]{anderson1973network} is replaced by Lemma \ref{lem:1} in here.
\end{remark}
\subsection{Lossless negative imaginary transfer function matrices}
This subsection introduces the concepts of negative imaginary and
lossless negative imaginary transfer matrices, and studies some important properties of lossless negative imaginary systems.
\begin{definition}\cite{liu2016non}
\label{def:3}
A square real-rational transfer function matrix $G(s)$ is said to be negative imaginary if
\begin{enumerate}
\item $G(s)$ has no poles in $\mathrm{Re}[s]>0$;\label{def:3-1}
\item $j[G(j\omega)-G^{*}(j\omega)]\geq0$ for all $\omega>0$ except values of $\omega$ where $j\omega$ is a pole of $G(s)$;\label{def:3-2}
\item If $s=0$ is a pole of $G(s)$, then it is at most a double pole, $\lim_{s\rightarrow 0}s^{2}G(s)$ is positive semidefinite Hermitian, and $\lim_{s\rightarrow 0}s^{m}G(s)=0$ for all $m\geq3$; \label{def:3-5}
\item If $s=j\omega_{0}$ with $\omega_{0}>0$ is a pole of $G(s)$, $\omega_{0}$ is finite, it is at most a simple pole and the residue matrix $K=\lim_{s\rightarrow j\omega_{0}}(s-j\omega_{0})jG(s)$ is positive semidefinite Hermitian;\label{def:3-3}
\item If $s=j\infty$ is a pole of $G(s)$, then it is at most a double pole, $\lim_{\omega \rightarrow \infty}\frac{G(j\omega)}{(j\omega)^{2}}$ is negative semidefinite Hermitian, and $\lim_{\omega \rightarrow \infty}\frac{G(j\omega)}{(j\omega)^{m}}=0$ for all $m\geq3$.\label{def:3-4}
\end{enumerate}
\end{definition}

\begin{definition}\cite{liu2016non}
\label{def:4}
A square real-rational transfer function matrix $G(s)$ is said to be lossless negative imaginary if
\begin{enumerate}
 \item $G(s)$ is negative imaginary;\label{def:1-1}
 \item $j[G(j\omega)-G^{*}(j\omega)]=0$ for all $\omega>0$ except values of $\omega$ where $j\omega$ is a pole of $G(s)$.\label{def:1-2}
\end{enumerate}
\end{definition}

The following lemma provides a necessary and sufficient condition for a system to be non-proper lossless negative imaginary.
\begin{lemma}
\label{lem:3}
A square real-rational transfer function matrix $G(s)$ is lossless negative imaginary if and only if
\begin{enumerate}
  \item All poles of elements of $G(s)$ are purely imaginary; \label{lem:1-1}
\item If $s=0$ is a pole of $G(s)$, it is at most a double pole,  $\lim_{s\rightarrow 0}s^{2}G(s)$ is positive semidefinite Hermitian, and $\lim_{s\rightarrow 0}s^{m}G(s)=0$ for all $m\geq3$; \label{lem:1-2}
\item If $s=j\omega_{0}$ with $\omega_{0}>0$ is a pole of $G(s)$, $\omega_{0}$ is finite, it is at most a simple pole and the residue matrix $K=\lim_{s\rightarrow j\omega_{0}}(s-j\omega_{0})jG(s)$ is positive semidefinite Hermitian;\label{lem:1-3}
\item If $s=j\infty$ is a pole of $G(s)$, it is at most a double pole, $\lim_{\omega \rightarrow \infty}\frac{G(j\omega)}{(j\omega)^{2}}$ is negative semidefinite Hermitian, and $\lim_{\omega \rightarrow \infty}\frac{G(j\omega)}{(j\omega)^{m}}=0$ for all $m\geq3$;\label{lem:1-4}
  \item $G(s)=G^{T}(-s)$ for all $s$ such that $s$ is not a pole of any element of $G(s)$.\label{lem:1-5}
\end{enumerate}
\end{lemma}
\begin{proof}
(Necessity) Suppose $G(s)$ is lossless negative imaginary. Condition \ref{def:1-2} of Definition \ref{def:4} implies that $j[G(j\omega)-G^{*}(j\omega)]=0$ for all $\omega>0$ except values of $\omega$ where $j\omega$ is a pole of $G(s)$. Then, we have $\overline{j[G(j\omega)-G^{*}(j\omega)]}=0$ for all $\omega>0$ with $j\omega$ not a pole of $G(s)$, that is, $j[G(j\omega)-G^{*}(j\omega)]=0$ for all $\omega<0$ with $j\omega$ not a pole of $G(s)$. According to the continuity of $G(s)$, it follows that $j[G(0)-G^{*}(0)]=0$. Hence, we have
\begin{equation*}
j[G(s)-G^{T}(-s)]=0,
\end{equation*}
for all $s=j\omega$, where $j\omega$ is not a pole of $G(s)$. Because
$j[G(s)-G^{T}(-s)]$ is an analytic function of $s$, it
follows from maximum modulus theorem
(\cite[Theorem A4-3]{newcomb1966linear}) that $j[G(s)-G^{T}(-s)]=0$ holds for all $s$ such that $s$ is not a pole of $G(s)$, and hence $G(s)=G^{T}(-s)$. Condition \ref{lem:1-5} holds.

Suppose $s_0$ is a pole of $G(s)$. It follows from Condition \ref{lem:1-5} that $-s_{0}$ is also a pole of $G(s)$. According to Definition \ref{def:3}, we know that $G(s)$ has no poles in $\mathrm{Re}[s]>0$. If $\mathrm{Re}[s_{0}]<0$, then $\mathrm{Re}[-s_{0}]>0$, there exists contradiction. So, the only case is that all poles of elements of $G(z)$ lie on the imaginary axis. Condition \ref{lem:1-1} holds.
Moreover, conditions 3-5 of Definition \ref{def:3} imply that conditions 2-4 hold.

(Sufficiency) Suppose conditions 1-5 hold. Conditions 1-4 imply Condition 1 and conditions 3-5 of Definition \ref{def:3} hold. Then, Condition 5 implies $G(s)=G^{T}(-s)$, which implies that $j[G(j\omega)-G^{*}(j\omega)]=0$ for all $\omega>0$ such that $j\omega$ is not a pole of any element of $G(s)$. It follows from Definitions \ref{def:3} and \ref{def:4} that $G(s)$ is lossless negative imaginary.
\end{proof}

\begin{remark}
Lemma \ref{lem:3} can be considered as a generalization of Lemma 2 in \cite{xiong2012lossless} by allowing poles at zero and infinity.
\end{remark}

The following lemma characterizes the properties of sum of non-proper lossless negative imaginary transfer matrices.
\begin{lemma}
Given two square real-rational lossless negative imaginary transfer function matrices $G_{1}(s)$ and $G_{2}(s)$, and a negative imaginary transfer function matrix $G_{3}(s)$. Then,
\begin{enumerate}
  \item $G_{1}(s)+G_{2}(s)$ is lossless negative imaginary;
   \item $G_{1}(s)+G_{3}(s)$ is negative imaginary;
\end{enumerate}
\end{lemma}
\begin{proof}
The proof is trivial according to the definitions of negative imaginary and lossless negative imaginary transfer function matrices.
\end{proof}

\subsection{Partial-fraction expansion of lossless negative imaginary systems}
In this subsection, we consider the minor decomposition theory of lossless negative imaginary systems in terms of a partial-fraction expansion, which provides the core to develop the lossless negative imaginary theory in this paper.

Suppose $G(s)$ is a square real-rational lossless negative imaginary transfer matrix. Then, define the following matrices,
\begin{align}
\label{eq:4-2}
\begin{aligned}
A_{2}&=\lim_{\omega \rightarrow \infty}\frac{G(j\omega)}{(j\omega)^{2}},\quad
C_2=\lim_{s\rightarrow0}s^2G(s),\\
A_1&=\lim_{\omega \rightarrow \infty}\frac{(G(j\omega)-(j\omega)^2A_2)}{j\omega},\\
C_1&=\lim_{s\rightarrow0}s(G(s)-\frac{C_2}{s^2}).
\end{aligned}
\end{align}
According to Lemma \ref{lem:3}, it follows that $A_{2}=A^{*}_{2}\leq0$ and $C_{2}=C^{*}_{2}\geq0$. Note that $j\omega$ is a pole of $G(s)$, the $-j\omega$ must also be a pole of $G(s)$, that is, $j\omega$ and $-j\omega$ occur in pairs. Because all the poles of lossless negative imaginary transfer matrices lie on the imaginary axis, and the forms for these poles in a partial-fraction expansion are known, $G(s)$ can be decomposed in terms of those residue matrix properties as the following form: 
\begin{align*}
G(s)=&\sum_{i}\frac{-jK_{i}}{s-j\omega_{i}}+\sum_{i}\frac{jK^{*}_{i}}{s+j\omega_{i}}
+\frac{1}{s}C_{1}+\frac{1}{s^{2}}C_{2}\\
&+sA_{1}+s^{2}A_{2}+G(\infty)\\
=&\sum_{i}\frac{sQ_{i}+T_{i}}{s^{2}+\omega^{2}_{i}}+\frac{1}{s}C_{1}
+\frac{1}{s^{2}}C_{2}+sA_{1}+s^{2}A_{2}+G(\infty),
\end{align*}
where $K_{i}$ is the residue matrix of $jG(s)$ at $j\omega_{i}$, $A_{2}=A^{*}_{2}\leq0$, $C_{2}=C^{*}_{2}\geq0$, $Q_{i}=j(K^{*}_{i}-K_{i})$, and $T_{i}=\omega_{i}(K_{i}+K^{*}_{i})$. According to Condition 3 of Lemma \ref{lem:3}, we know $K_{i}=K^{*}_{i}\geq0$, it follows that $Q_{i}=0$ and $T_{i}=T^{*}_{i}$. Then, we have
\begin{equation*}
G(s)=\sum_{i}\frac{T_{i}}{s^{2}+\omega^{2}_{i}}+\frac{1}{s}C_{1}
+\frac{1}{s^{2}}C_{2}+sA_{1}+s^{2}A_{2}+G(\infty).
\end{equation*}
We can find that $\sum_{i}\frac{T_{i}}{s^{2}+\omega^{2}_{i}}$, $\frac{1}{s^{2}}C_{2}$ and
$s^{2}A_{2}$ are lossless negative imaginary. The fact that they are negative imaginary is immediate from the definition of negative imaginary systems, the lossless character follows by observing that
\begin{align*}
&j\big[\frac{1}{(j\omega)^{2}}C_{2}-\frac{1}{(-j\omega)^{2}}C^{*}_{2}\big]=0;\\
&j[(j\omega)^{2}A_{2}-(-j\omega)^{2}A^{*}_{2}]=0;\\
&j\big[\sum_{i}\frac{T_{i}}{(j\omega)^{2}+\omega^{2}_{i}}
-\sum_{i}\frac{T^{*}_{i}}{(-j\omega)^{2}+\omega^{2}_{i}}\big]=0.
\end{align*}


Now, we will study the properties of matrices $A_{1}$ and $C_{1}$ in the following lemma, where $A_1$ and $C_1$ are defined in \eqref{eq:4-2}.
\begin{lemma}
\label{lem:5}
Given a square real-rational lossless negative imaginary transfer function matrix $G(s)$.
Then $A_1+A^{*}_1=0$ and $C_1+C^{*}_1=0$ hold.
\end{lemma}
\begin{proof}
Suppose $G(s)$ is lossless negative imaginary. It follows that $G(s)$ has at most a double pole at infinity and zero.
First, we will prove $A_1+A^{*}_1=0$. When $G(s)$ has no poles at infinity, one has that $A_2=0$, $A_1=0$, and hence $A_1+A^{*}_1=0$.

When $G(s)$ has a simple pole at infinity, $A_2=0$. Let
\begin{equation*}
G(s)=G_{0}(s)+sA_1+G(\infty),
\end{equation*}
 where $G_{0}(s)$ is strictly proper. $G(s)$ and $G_{0}(s)$ have the same poles except at infinity. Condition \ref{def:1-2} of Definition \ref{def:4} implies that
\begin{align}
\label{eq:2}
\begin{aligned}
j[G(j\omega)&-G^{*}(j\omega)]\\
=&j[j\omega A_1+G_{0}(j\omega)+G^{T}(\infty)+j\omega A^{*}_1\\
&-G^{*}_{0}(j\omega)-G^{*}(\infty)]\\
=&-\omega(A_1+A^{*}_{1})+j[G_{0}(j\omega)-G^{*}_{0}(j\omega)]=0,
\end{aligned}
\end{align}
for all $\omega>0$, where $j\omega$ is not a pole of $G(s)$ and $G_{0}(s)$.
Suppose $A_1+A^{*}_{1}\neq0$. Then, since $G_{0}(s)$ is strictly proper, there exists a sufficiently large $\omega_1$ such that $j[G_{0}(j\omega_1)-G^{*}_{0}(j\omega_1)]=0$, which contradicts with \eqref{eq:2}. So, $A_1+A^{*}_{1}=0$.

When $G(s)$ has a double pole at infinity, Let
\begin{equation*}
 G(s)=G_{0}(s)+sA_1+s^{2}A_2+G(\infty),
\end{equation*}
where $A_2=A^{*}_2\leq0$ and $G_{0}(s)$ is strictly proper. Similarly, Condition \ref{def:1-2} of Definition \ref{def:4} implies that $
j[G(j\omega)-G^{*}(j\omega)]
=-\omega(A_1+A^{*}_{1})+j[G_{0}(j\omega)-G^{*}_{0}(j\omega)]=0
$. Using the similar analysis as the case where $G(s)$ has a simple pole at infinity, we have $A_1+A^{*}_{1}=0$.

Next, we will prove $C_1+C^{*}_{1}=0$. When $G(s)$ has no poles at zero, one has that $C_2=0$, $C_1=0$, and hence $C_1+C^{*}_1=0$.

When $G(s)$ has a simple pole at zero, $C_2=0$. Let
\begin{equation*}
 G(s)=\sum_{i}\frac{T_{i}}{s^{2}+\omega^{2}_{i}}+\frac{1}{s}C_1+sA_1+s^{2}A_2+G(\infty),
 \end{equation*}
where $T_{i}=T^{*}_{i}$. Condition \ref{def:1-2} of Definition \ref{def:4} implies that
\begin{equation*}
\label{eq:3}
\begin{array}{l}
j[G(j\omega)-G^{*}(j\omega)]\\
=j\bigg[\sum_{i}\frac{T_{i}}{(j\omega)^{2}+\omega^{2}_{i}}+\frac{1}{j\omega}C_1+(j\omega)A_1+(j\omega)^{2}A_2+G(\infty)\\
-\sum_{i}\frac{T^{*}_{i}}{(-j\omega^{2})+\omega^{2}_{i}}-\frac{1}{(-j\omega)}C_1-(-j\omega)A_1
-(-j\omega)^{2}A_2-G^{*}(\infty)\bigg]\\
=-\omega(A_1+A^{*}_{1})+\frac{1}{\omega}(C_1+C^{*}_{1})]=0,
\end{array}
\end{equation*}
for all $\omega>0$, where $j\omega$ is not a pole of $G(s)$. Because $-\omega(A_1+A^{*}_{1})=0$, it follows that $C_1+C^{*}_{1}=0$.

When $G(s)$ has a double pole at zero, let
\begin{equation*}
 G(s)=\sum_{i}\frac{T_{i}}{s^{2}+\omega^{2}_{i}}+\frac{1}{s}C_1+\frac{1}{s^{2}}C_2+sA_1+s^{2}A_2+G(\infty),
 \end{equation*}
 where $T_{i}=T^{*}_{i}$. Then, Condition \ref{def:1-2} of Definition \ref{def:4}
 with $A_1+A^{*}_1=0$
 implies that
\begin{equation*}
j[G(j\omega)-G^{*}(j\omega)]
=\frac{1}{\omega}(C_1+C^{*}_{1})]=0.
\end{equation*}
 This completes the proof.
\end{proof}
\begin{remark}
If $G(s)$ is a symmetric lossless negative imaginary transfer matrix, it follows that $A_{1}=A^{*}_{1}\leq0$ and $C_1=C^{*}_{1}\geq0$, and hence, $A_1=0$ and $C_1=0$. In other words, it is impossible for symmetric lossless negative imaginary transfer matrix having  simple poles at zero and infinity.
\end{remark}

The following lemma gives a decomposed property about lossless negative imaginary transfer matrices.
\begin{lemma}
\label{lem:6}
Let $G(s)$ be a square real-rational transfer function matrix of the form
\begin{equation*}
G(s)=G_{0}(s)+\frac{1}{s}C_{1}+\frac{1}{s^{2}}C_{2}+sA_{1}+s^{2}A_{2}+G(\infty),
\end{equation*}
where $G_{0}(s)$ is strictly proper, and $G_{0}(s)$ has no poles at zero and infinity. Then, $G(s)$ is lossless negative imaginary if and only if $G_{0}(s)$ is lossless negative imaginary, $A_{2}=A^{*}_{2}\leq0$, $C_{2}=C^{*}_{2}\geq0$, $A_{1}+A^{*}_{1}=0$, $C_{1}+C^{*}_{1}=0$ and $G(\infty)=G^{T}(\infty)$.
\end{lemma}
\begin{proof}
(Necessity) Suppose $G(s)$ is lossless negative imaginary. According to Lemmas \ref{lem:3} and \ref{lem:5}, it follows that $A_{2}=A^{*}_{2}\leq0$, $C_{2}=C^{*}_{2}\geq0$, $A_{1}+A^{*}_{1}=0$, $C_{1}+C^{*}_{1}=0$ and $G(\infty)=G^{T}(\infty)$. $G(s)$ and $G_{0}(s)$ have the same poles except at zero and infinity. For $\omega>0$, $j\omega$ is not a pole of $G(s)$ and $G_{0}(s)$, we have
\begin{equation*}
j[G(j\omega)-G^{*}(j\omega)]=j[G_{0}(j\omega)-G^{*}_{0}(j\omega)]=0.
\end{equation*}
If $j\omega_{0}$, $\omega_{0}>0$ is a pole of $G(s)$, then $\lim_{s\rightarrow j\omega_{0}}(s-j\omega_{0})jG(s)=\lim_{s\rightarrow j\omega_{0}}(s-j\omega_{0})jG_{0}(s)$. Hence, $G_{0}(s)$ is lossless negative imaginary.

(Sufficiency) Suppose $G_{0}(s)$ is lossless negative imaginary, and $A_{2}=A^{*}_{2}\leq0$, $C_{2}=C^{*}_{2}\geq0$, $A_{1}+A^{*}_{1}=0$, $C_{1}+C^{*}_{1}=0$ and $G(\infty)=G^{T}(\infty)$. It follows that $\frac{1}{s}C_{1}$, $\frac{1}{s^{2}}C_{2}$, $sA_{1}$, $s^{2}A_{2}$ are lossless negative imaginary. Then, according to the sum properties of lossless negative imaginary systems, $G(s)$ is lossless negative imaginary.
\end{proof}

\begin{remark}
Based on the analysis in this subsection, we can find that the lossless negative imaginary transfer matrices can be seen as a sum of several lossless negative imaginary transfer matrices. For example, consider $G(s)=\frac{1-2s^{4}}{s^{2}(s^2+1)}$. $G(s)$ can be decomposed as $G(s)=\frac{1}{s^{2}}+\frac{1}{s^2+1}-2$, where $C_{2}=1$, $G(\infty)=-2$, the residue matrix of $jG(s)$ at $s=j$ is $K=\frac{1}{2}$, and hence $T_{1}=K+K^{*}=1$. Both $\frac{1}{s^2}$ and $\frac{1}{s^2+1}$ are lossless negative imaginary. Moreover, the negative imaginary transfer matrices also have similar properties as Lemma \ref{lem:6}. We are able to decompose any negative imaginary transfer matrix $G(s)$ into the sum of a lossless negative imaginary transfer matrix $G_{LNI}(s)$ and a negative imaginary transfer matrix $G_{NI}(s)$.
\end{remark}

Decompose the lossless negative imaginary transfer matrices into proper part and non-proper part. We have the following result.
\begin{corollary}
\label{co:1}
Let $G(s)$ be a square real-rational transfer function matrix of the form
\begin{equation*}
G(s)=G_{0}(s)+sA_{1}+s^{2}A_{2}+\sum_{i}s^{i}A_{i},
\end{equation*}
where $G_{0}(s)$ has no poles at infinity. Then, $G(s)$ is lossless negative imaginary if and only if $G_{0}(s)$ is lossless negative imaginary, $A_{2}=A^{*}_{2}\leq0$, $A_{1}+A^{*}_{1}=0$ and $A_{i}=0$ for $i\geq3$.
\end{corollary}
\begin{proof}
Trivial.
\end{proof}
\begin{remark}
Corollary \ref{co:1} is useful in deriving the non-proper descriptor lossless negative imaginary lemma. Also, for the non-proper negative imaginary transfer matrices, we have a similar result, that is, $G(s)$ is negative imaginary if and only if $A_{2}=A^{*}_{2}\leq0$ and $G_{0}(s)+sA_{1}$ is negative imaginary, which is also useful in the study of non-proper descriptor negative imaginary systems. If $G(s)$ is symmetric, we have $G(s)$ is negative imaginary if and only if $A_{2}=A^{*}_{2}\leq0$, $A_{1}=A^{*}_{1}\leq0$ and $G(s)$ is negative imaginary.
\end{remark}

\section{Relationship between lossless positive real and lossless negative imaginary systems}
\label{sec:relation}
In this section, two new relationships between lossless negative imaginary and lossless positive real transfer functions will be established in non-proper and proper case. First, we present a new description of the relationship between non-proper lossless negative imaginary and non-proper lossless positive real transfer matrices.
\begin{lemma}
\label{lem:7}
Let $G(s)$ be a square real-rational transfer function matrix. Suppose $G(s)$ has no poles at zero. Then $G(s)$ is lossless negative imaginary if and only if
\begin{enumerate}
  \item $G(0)=G^{T}(0)$;
  \item $F(s)=-\frac{1}{s}[G(s)-G(0)]$ is lossless positive real.
\end{enumerate}
\end{lemma}
\begin{proof}
(Necessity)
Suppose $G(s)$ is lossless negative imaginary.
It follows from \cite[Lemma 9]{liu2016non} that $G(0)=G^{T}(0)$.
When $G(s)$ has no pole at infinity, $F(s)$ has no poles at infinity. When $G(s)$ has a simple pole at infinity, then $F(s)$ has also no poles at infinity. Let $G(s)=sA_{1}+G_{0}(s)$, where $G_{0}(s)$ is proper and $A_{1}+A^{T}_{1}=0$. Then,
\begin{equation*}
F(s)=-A_1-\frac{1}{s}G_{0}(s)+\frac{1}{s}G(0).
\end{equation*}
As $\omega\rightarrow \infty$,
 it follows that
\begin{equation*}
F(j\omega)+F^{*}(j\omega)=-(A_{1}+A^{T}_1)=0.
\end{equation*}
The rest of the proof is the same as the necessity proof of
\cite[Lemma 9]{liu2016non}.

(Sufficiency) The sufficient proof is the same as the sufficient proof of \cite[Lemma 9]{liu2016non}.
\end{proof}
\begin{example}
To illustrate the usefulness of Lemma \ref{lem:7}, consider a non-proper transfer matrix $G(s)=\left(
                                 \begin{array}{cc}
                                   \frac{1}{s^{2}+1} & -s \\
                                   s & \frac{1}{s^2+1} \\
                                 \end{array}
                               \right)
$. $G(s)$ has no poles in $\mathrm{Re}[s]>0$. A calculation shows that $j[G(j\omega)-G^{*}(j\omega)]=0$. $G(s)$ has a simple pole at infinity and $s=j$. The residue matrix of $jG(s)$ at $s=j$ is positive semidefinite Hermitian, being %
 $K=\lim_{s\rightarrow j}(s-j)jG(s)=\left(
                                          \begin{array}{cc}
                                            \frac{1}{2} & 0 \\
                                            0 & \frac{1}{2} \\
                                          \end{array}
                                        \right)
$. Moreover, $\lim_{\omega\rightarrow \infty}\frac{G(j\omega)}{(j\omega)^{2}}=0$ and $A_1=\lim_{\omega\rightarrow \infty}\frac{G(j\omega)}{j\omega}=\left(
                                      \begin{array}{cc}
                                        0 & -1 \\
                                        1 & 0 \\
                                      \end{array}
                                    \right)
$, which satisfies $A_{1}+A^{*}_{1}=0$. According to Definitions \ref{def:3} and \ref{def:4}, it follows that $G(s)$ is lossless negative imaginary. We can say that $G(s)$ is lossless negative imaginary if and only if $F(s)=-\frac{1}{s}[G(s)-G(0)]=\left(
                                \begin{array}{cc}
                                  \frac{s}{s^{2}+1} & 1 \\
                                  -1 & \frac{s}{s^{2}+1} \\
                                \end{array}
                              \right)
$ is lossless positive real and $G(0)=G^{T}(0)$. A calculation shows that $F(s)$ satisfies all conditions in Definition \ref{def:2}: $F(s)$ is positive real, and $F(j\omega)+F^{*}(j\omega)=0$ for all $\omega$ with $j\omega$ not a pole of $F(s)$.
\end{example}

%
When $G(s)$ is a real-rational proper transfer matrix, we have the following result.

\begin{lemma}
\label{lem:8}
Let $G(s)$ be a square real-rational proper transfer function matrix. Then $G(s)$ is lossless negative imaginary if and only if
\begin{enumerate}
  \item $G(\infty)=G^{T}(\infty)$;
  \item $F(s)=s[G(s)-G(\infty)]$ is lossless positive real.
\end{enumerate}
\end{lemma}
\begin{proof}
(Necessity)
Suppose $G(s)$ is lossless negative imaginary.
It follows from \cite[Lemma 11]{liu2016non} that $G(\infty)=G^{T}(\infty)$.
When $G(s)$ has no pole at zero, $F(s)$ has no poles at zero, and $F(0)+F^{*}(0)=0$. When $G(s)$ has a simple pole at zero, then $F(s)$ has also no poles at zero. Let $G(s)=\frac{1}{s}C_1+G_{0}(s)$, where $G_{0}(s)$ has no poles at zero and $C_{1}+C^{*}_{1}=0$. Then,
$F(s)=C_1+sG_{0}(s)-sG(\infty)$,
 and hence,
$F(0)+F^{*}(0)=C_{1}+C^{T}_1=0$.
The rest of the proof is the same as the necessity proof of \cite[Lemma 11]{liu2016non}.

(Sufficiency) The sufficient proof is the same as the sufficient proof of \cite[Lemma 11]{liu2016non}.
\end{proof}
\begin{example}
As an illustration of Lemma \ref{lem:8}, consider a proper transfer matrix $G(s)=\left(
               \begin{array}{cc}
                 \frac{-s^2}{s^2+1} & \frac{1}{s}+1 \\
                 \frac{-1}{s}+1 & \frac{-s^2}{s^2+1} \\
               \end{array}
             \right)
$. It can be found that $G(s)$ is lossless negative imaginary if and only if $F(s)=s[G(s)-G(\infty)]=\left(
                          \begin{array}{cc}
                            \frac{s}{s^2+1} & 1 \\
                            -1 & \frac{s}{s^2+1} \\
                          \end{array}
                        \right)
$ is lossless positive real and $G(\infty)=G^{T}(\infty)$. A calculation shows that $G(s)$ and $F(s)$ satisfy all conditions in Definition \ref{def:4} and Definition \ref{def:2}, respectively. Note that $G(s)$ has a simple pole at zero, $C_{1}=\lim_{s\rightarrow0}sG(s)=\left(
                                   \begin{array}{cc}
                                     0 & 1 \\
                                     -1 & 0 \\
                                   \end{array}
                                 \right)
$ satisfies $C_{1}+C^{T}_{1}=0$, and $F(s)$ has no poles at zero.
\end{example}
\begin{remark}
Lemma \ref{lem:7} can be considered as a generalization of Lemma 9 in \cite{liu2016non} by allowing simple pole at infinity. Lemma \ref{lem:8} can be considered as a generalization of Lemma 11 in \cite{liu2016non} by allowing simple pole at zero.
\end{remark}

\section{Lossless negative imaginary lemma in state-space conditions}
\label{sec:lemma}
The lossless negative imaginary lemma proposed in this section, which is the main result of the paper, extends the lossless negative imaginary lemma in \cite{xiong2012lossless} to the case where the transfer
matrices may have poles at zero. Theorem \ref{th:1} could be considered as a modification of Lemma 2 in \cite {mabrok2015generalized} applied to lossless negative imaginary case.
\begin{theorem}
\label{th:1}
Let $(A,B,C,D)$ be a minimal state-space realization of a square real-rational proper transfer function matrix $G(s)\in\mathcal{R}^{m\times m}$, where $A\in\mathbb{R}^{n\times n}$, $B\in\mathbb{R}^{n\times m}$, $C\in\mathbb{R}^{m\times n}$, $D\in\mathbb{R}^{m\times m}$, and $m\leq n$. Then, $G(s)$ is lossless negative imaginary if and only if $D=D^{T}$, and there exists a real matrix $P=P^{T}\geq0$, $P\in\mathbb{R}^{n\times n}$, such that
\begin{align}
\label{eq:7}
\left(
  \begin{array}{cc}
    PA+A^{T}P & PB-A^{T}C^{T} \\
    B^TP-CA & CB+(CB)^{T} \\
  \end{array}
\right)=0.
\end{align}

\end{theorem}

\begin{proof}
 The equivalence follows along the following sequence of equivalent reformulations.

$G(s)\sim (A,B,C,D)$ is lossless negative imaginary.

$\Leftrightarrow$ $G(\infty)=G^{T}(\infty)$, and $F(s)=s[G(s)-G(\infty)]$ is lossless positive real (according to Lemma \ref{lem:8}).

$\Leftrightarrow$ $D=D^{T}$, and $F(s)\sim (A,B,CA,CB)$ is lossless positive real. Note that $(A,B)$ is completely controllable and $(A,CA)$ may be not observable. The reason is that $A$ may be singular.

$\Leftrightarrow$ $D=D^{T}$, and there exists a real matrix $P=P^{T}\geq0$, $P\in\mathbb{R}^{n\times n}$, such that
\begin{align*}
PA+A^TP&=0\\
PB-A^{T}C^{T}&=0\\
CB+(CB)^{T}&=0.
\end{align*}
This equivalence is according to Lemma \ref{lem:2}.

$\Leftrightarrow$ $D=D^{T}$, and there exists a real matrix $P=P^{T}\geq0$, $P\in\mathbb{R}^{n\times n}$, such that \eqref{eq:7} holds.
\end{proof}

\begin{remark}
In Theorem \ref{th:1}, the state-space realization $(A,B,C,D)$ is assumed to be minimal realization. In fact, if we remove the observability requirement of $(A,C)$, the results in Theorem \ref{th:1} also hold. Moreover, consider the generalized negative imaginary lemma in \cite[Lemma 2]{mabrok2015generalized}. Assume that $(A,B)$ is controllable and $(A,C)$ is not necessarily observable. The result in \cite[Lemma 2]{mabrok2015generalized} also hold by using Lemma 10 in \cite{liu2016non} and Lemma \ref{lem:1} in here.
\end{remark}
\begin{remark}
Compared to Theorem 1 in \cite{xiong2012lossless}, Theorem \ref{th:1} in this paper removes the non-singularity condition of state matrix $A$, that is, $\det(A)=0$ is allowed in this paper by allowing poles at zero, and $P$ is allowed to be positive semi-definite.
Compared to Lemma 2 in \cite{mabrok2015generalized}, the inequality in \cite{mabrok2015generalized} is modified as equality in this paper by applying to the case where the negative imaginary transfer matrix is lossless.
\end{remark}
\begin{example}
To illustrate the main results of the paper, consider the same example in Remark 4.
One minimal realization of $G(s)=\frac{1-2s^{4}}{s^{2}(s^2+1)}$ is as follows
\begin{align*}
A&=\left(
    \begin{array}{cccc}
      0 & -1 & 0 & 0 \\
      1 & 0 & 0 & 0 \\
      0 & 1 & 0 & 0 \\
      0 & 0 & 1 & 0 \\
    \end{array}
  \right),\quad B=\left(
              \begin{array}{c}
                1 \\
                0 \\
                0 \\
                0 \\
              \end{array}
            \right),\\
            C&=\left(
                \begin{array}{cccc}
                  0 & 2 & 0 & 1 \\
                \end{array}
              \right),\quad D=-2.
\end{align*}
$CB+(CB)^{T}=0$ holds.
YALMIP \cite{lofberg2004yalmip} and SeDuMi were used to find a
solution of \eqref{eq:7} as
\begin{equation*}
P=\left(
    \begin{array}{cccc}
      2 & 0 & 1 & 0 \\
      0 & 1 & 0 & 0 \\
      1 & 0 & 1 & 0 \\
      0 & 0 & 0 & 0 \\
    \end{array}
  \right)\geq0,
\end{equation*}
which implies that the conditions in Theorem \ref{th:1} hold. This verifies that $G(s)$ is lossless negative imaginary from state-space condition. Also, the lossless negative imaginary property of $G(s)$ can also be confirmed by directly using the definition of lossless negative imaginary systems.
\end{example}

%
\section{Conclusions}
\label{sec:conclusions}

This paper has studied some new and important lossless negative imaginary properties of square real-rational transfer matrices. A necessary and sufficient condition has been proposed to characterize the non-proper lossless negative imaginary systems. Then, a minor decomposition theory for lossless negative imaginary systems has been proposed in terms of a partial-fraction expansion. According to this minor decomposition theory,
two different relationships between lossless negative imaginary and lossless positive real transfer matrices have been studied by allowing poles at zero and infinity. Moreover, a generalized lossless negative imaginary lemma has been derived in terms of a minimal realization.

%
%
%
%


\end{document}